\documentclass[a4paper,10pt,reqno]{amsart}
\input{preamble.sty}

\usepackage[T1]{fontenc}
\usepackage[utf8]{inputenc}
\usepackage{palatino}
\usepackage{textcomp}

\graphicspath{{figures/}}
\usepackage{float}

\usepackage{cleveref}
\usepackage[backend=bibtex,maxbibnames=99,firstinits=true,style=ieee,sorting=nty,style=numeric-comp]{biblatex}
\DeclareFieldFormat{postnote}{#1}
\newcommand{\ourpg}{\Sigma_{1,2}}

\allowdisplaybreaks

\title{Stein-fillable open books of genus one that do not admit positive factorisations}
\author{Vitalijs Brejevs and Andy Wand}
\address{School of Mathematics and Statistics, University of Glasgow, Glasgow, UK}
\email{Vitalijs.Brejevs@glasgow.ac.uk}
\address{School of Mathematics and Statistics, University of Glasgow, Glasgow, UK}
\email{Andy.Wand@glasgow.ac.uk}

\begin{document}
	\begin{abstract}
		We construct an infinite family of genus one open book decompositions supporting Stein-fillable contact structures and show that their monodromies do not admit positive factorisations. This extends a line of counterexamples in higher genera and establishes that a correspondence between Stein fillings and positive factorisations only exists for planar open book decompositions.
	\end{abstract}
	\maketitle
	
	\section{Introduction}
	\label{sec:intro}

	In the foundational work~\cite{giroux}, Giroux has established a one-to-one correspondence between isotopy classes of contact structures on a 3-manifold $ Y $ and positive stabilisation classes of open book decompositions of $ Y $, enabling one to consider questions of contact and symplectic geometry through a powerful lens of surface mapping class groups. In particular, a natural question when studying contact manifolds is that of fillability, i.e., determining when a contact manifold can be the boundary of a symplectic manifold in some compatible way; in this paper, we are concerned with Stein fillability. Results of Giroux coupled with the work of Loi and Piergallini~\cite{LoiPiergallini}, Akbulut and Özbağcı~\cite{akbulutozbagci} and Plamenevskaya~\cite{plam-stein} drew a further connection between the worlds of surface diffeomorphisms and symplectic geometry, establishing that a contact manifold is Stein-fillable if and only if the monodromy of \emph{some} open book supporting it admits a positive factorisation into Dehn twists. The picture, however, is still complicated: for example, proving that a contact manifold is not Stein-fillable this way entails the usually intractable task of obstructing positive factorisability of \emph{all} monodromies of supporting open books.

	A tempting but untrue strengthening of this result would be the claim that the monodromy of \emph{every} open book $ (\Sigma, \varphi) $ supporting a Stein-fillable contact manifold $ (Y, \xi) $ factorises positively. Indeed, a result of Wendl~\cite{wendl} implies that if the genus $ g(\Sigma) = 0 $, then Stein fillings of $ (Y, \xi) $, up to symplectic deformation, are in one-to-one correspondence with positive factorisations of $ \varphi $, up to conjugation. However, if $ g(\Sigma) \geqslant 2 $, it follows from the work of the second author~\cite{wand2015} and Baker, Etnyre and Van Horn-Morris~\cite{BEvHM} that $ \varphi $ need not admit any positive factorisation. The case of $ g(\Sigma) = 1 $ has been previously studied by Lisca~\cite{lisca-stein} who has shown that if $ \Sigma $ has one boundary component and $ Y $ is a Heegaard Floer $ L $-space, then $ (Y, \xi) $ is Stein-fillable if and only if $ \varphi $ admits a positive factorisation. The purpose of this paper is to exhibit for the first time a family of Stein-fillable contact manifolds supported by open books with $ g(\Sigma) = 1$ whose monodromies do not factorise positively.

	\begin{thm}
	Let $ n \geqslant 0 $. Then $ (\Sigma_{1,2}, \varphi_n) $ with $ \varphi_n = \tau_\alpha \tau_\beta \tau_\gamma^{-1} \tau_{\delta_1} \tau_{\delta_2}^{4 + n} $, as illustrated in Figure~\ref{fig:monodromy-intro}, is an open book decomposition supporting a Stein-fillable contact manifold, but $ \varphi_n $ does not admit a positive factorisation. \vspace{-2em}
	\end{thm}

	\begin{center}
		\begin{figure}[h]
			\centering
			\def\svgwidth{0.5\textwidth}
			\makebox[\textwidth][c]{
\begingroup%
  \makeatletter%
  \providecommand\color[2][]{%
    \errmessage{(Inkscape) Color is used for the text in Inkscape, but the package 'color.sty' is not loaded}%
    \renewcommand\color[2][]{}%
  }%
  \providecommand\transparent[1]{%
    \errmessage{(Inkscape) Transparency is used (non-zero) for the text in Inkscape, but the package 'transparent.sty' is not loaded}%
    \renewcommand\transparent[1]{}%
  }%
  \providecommand\rotatebox[2]{#2}%
  \newcommand*\fsize{\dimexpr\f@size pt\relax}%
  \newcommand*\lineheight[1]{\fontsize{\fsize}{#1\fsize}\selectfont}%
  \ifx\svgwidth\undefined%
    \setlength{\unitlength}{138.86382456bp}%
    \ifx\svgscale\undefined%
      \relax%
    \else%
      \setlength{\unitlength}{\unitlength * \real{\svgscale}}%
    \fi%
  \else%
    \setlength{\unitlength}{\svgwidth}%
  \fi%
  \global\let\svgwidth\undefined%
  \global\let\svgscale\undefined%
  \makeatother%
  \begin{picture}(1,0.71190169)%
    \lineheight{1}%
    \setlength\tabcolsep{0pt}%
    \put(0,0){\includegraphics[width=\unitlength,page=1]{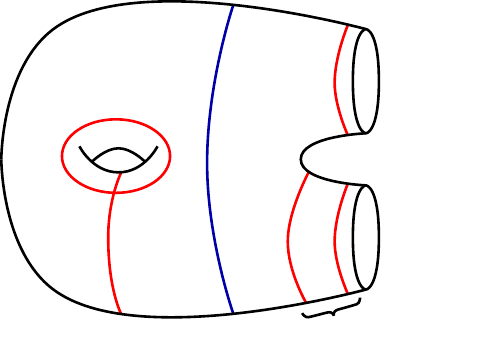}}%
    \put(0.22315832,0.48703189){\color[rgb]{0,0,0}\makebox(0,0)[lt]{\lineheight{1.25}\smash{\begin{tabular}[t]{l}$\alpha$\end{tabular}}}}%
    \put(0.25114554,0.18703241){\color[rgb]{0,0,0}\makebox(0,0)[lt]{\lineheight{1.25}\smash{\begin{tabular}[t]{l}$\beta$\end{tabular}}}}%
    \put(0.44668517,0.37582025){\color[rgb]{0,0,0}\makebox(0,0)[lt]{\lineheight{1.25}\smash{\begin{tabular}[t]{l}$\gamma$\end{tabular}}}}%
    \put(0.81284698,0.55429826){\color[rgb]{0,0,0}\makebox(0,0)[lt]{\lineheight{1.25}\smash{\begin{tabular}[t]{l}$\delta_1$\end{tabular}}}}%
    \put(0.81284698,0.18703179){\color[rgb]{0,0,0}\makebox(0,0)[lt]{\lineheight{1.25}\smash{\begin{tabular}[t]{l}$\delta_2$\end{tabular}}}}%
    \put(0,0){\includegraphics[width=\unitlength,page=2]{monodromy-intro.pdf}}%
    \put(0.65366342,0.00529196){\color[rgb]{0.15686275,0.04313725,0.04313725}\makebox(0,0)[lt]{\lineheight{1.25}\smash{\begin{tabular}[t]{l}$4+n$\end{tabular}}}}%
  \end{picture}%
\endgroup%
}
			\caption{
				An open book decomposition $ (\Sigma_{1,2}, \tau_\alpha \tau_\beta \tau_\gamma^{-1} \tau_{\delta_1} \tau_{\delta_2}^{4 + n}) $ with $ n \geqslant 0 $.
			} \label{fig:monodromy-intro}
		\end{figure}
	\end{center}

	\begin{rem}
	The open books in our examples have pages $ \Sigma_{1,2} $ with two boundary components, and we note that one can add 1-handles to them to obtain any surface $ \Sigma_{g,n} $ with $ g, n \geqslant 1 $ other than $ \Sigma_{1,1} $. Since adding a 1-handle to the page of an open book amounts to, on the level of 3-manifolds, taking a contact connected sum with $ S^1 \times S^2 $ endowed with its unique Stein-fillable contact structure, it also preserves Stein fillability. Moreover, if one attaches a 1-handle while extending the monodromy by the identity on the co-core of the 1-handle, it does not change the property of not being positively factorisable (cf.~\cite[Remark 5.3]{lisca-stein}). Hence, this leaves open only the case of open books whose pages are one-holed tori.
	\end{rem}




	The structure of the paper is as follows. In Section~\ref{sec:defs} we recall some pertinent facts about open book decompositions and their interplay with contact structures, as well as the notion of Stein fillability. In Section~\ref{sec:trefoil} we construct, via transverse contact surgery, an infinite family of Stein-fillable contact manifolds with explicitly given genus one open books. Finally, in Section~\ref{sec:non-positive} we show that the monodromies of those open books do not admit positive factorisations. \\

	\noindent \textbf{Acknowledgements.} We would like to thank Jeremy Van Horn-Morris for directing our attention towards this problem, and Brendan Owens for many helpful conversations. The first author was supported by the Carnegie Trust.

	\section{Basic definitions}
	\label{sec:defs}

	Recall that an \emph{open book decomposition} (or just \emph{open book}) of a closed 3-manifold $ Y $ is a pair $ (L, \pi) $ where $ L \subset Y $ is an oriented link, called the $ \emph{binding} $, and $ \pi : Y \setminus L \ra S^1 $ is a fibration such that $ \pi^{-1}(s) $ for any $ s \in S^1 $ is the interior of a compact orientable surface $ \Sigma_\pi $, called the \emph{page}, with $ \pp \Sigma_\pi = L $. Now, given a compact oriented surface $ \Sigma $ with boundary, denote by $ \Gamma_\Sigma $ the {\emph{mapping class group} of $ \Sigma $ consisting of isotopy classes of orientation-preserving self-diffeomorphisms of $ \Sigma $ that restrict to the identity on $ \pp \Sigma $; we shall confuse classes in $ \Gamma_\Sigma $ with their representatives. Any locally trivial bundle over oriented $ S^1 $ with the fibre $ \Sigma $ is canonically diffeomorphic to the fibration $ M_{\varphi} \ra S^1 $ for $ M_{\varphi} = [0, 1] \times \Sigma / (0, \varphi(x)) \sim (1, x) $ and $ \varphi $ an orientation-preserving self-diffeomorphism of $ \Sigma $, taken up to conjugation. Hence, an open book decomposition $ (L, \pi) $ of a 3-manifold $ Y_{(L, \pi)} $ determines a map $ \varphi_\pi \in \Gamma_{\Sigma_\pi} $, called the \emph{monodromy}. On the other hand, given a pair $ (\Sigma, \varphi) $ with $ \varphi \in \Gamma_\Sigma $ and $ \pp \Sigma \neq \varnothing $, we may construct a closed 3-manifold $ Y_{(\Sigma, \varphi)} $ in the following way: take the mapping torus $ M_\varphi $, identify all its boundary components with $ \bigsqcup_n S^1 \times S^1 $ for some $ n > 0 $, where in each $ S^1 \times S^1 $ the first factor comes from the quotient of the unit interval and the second from $ \pp \Sigma $, and glue in solid tori $ \bigsqcup_n D^2 \times S^1 $ via the identity map $ \bigsqcup_n \pp D^2 \times S^1 \ra \bigsqcup_n S^1 \times S^1 $ . The resultant $ Y_{(\Sigma, \varphi)} $ admits an open book decomposition with the binding given by the cores $ \bigsqcup_n \{ 0 \} \times S^1 $ of the solid tori, the page $ \Sigma $ and monodromy $ \varphi $. Hence, we can pass between $ (L, \pi) $ and $ (\Sigma, \varphi) $ to determine an open book decomposition of a closed 3-manifold up to diffeomorphism.

	Given $ \varphi \in \Gamma_\Sigma $, we say that $ \varphi $ admits a \emph{positive factorisation} if it can be written as a product of positive Dehn twists about essential simple closed curves in $ \Sigma $. We will denote by $ \Gamma^+_\Sigma $ the sub-monoid of $ \Gamma_\Sigma $ consisting of isotopy classes of positively factorisable maps.


	Recall also that a \emph{(positive) contact structure} on $ Y $ is an oriented plane field $ \xi \subset TY $ given by $ \ker \alpha $ for some 1-form $ \alpha \in \Omega^1(Y) $ satisfying $ \alpha \wedge \dd \alpha > 0 $. We say that $ \xi $ is \emph{supported} by an open book decomposition of $ Y $ if $ \alpha > 0 $ on the binding and $ \dd \alpha > 0 $ on the interior of the pages. In fact, every open book decomposition of $ Y $ supports some contact structure~\cite{thurston-wink}. Moreover, as recounted in Section~\ref{sec:intro}, Giroux has shown~\cite{giroux} that there exists a one-to-one correspondence between contact structures on $ Y $ up to contact isotopy and open book decompositions of $ Y $ up to \emph{positive stabilisation}, i.e., up to adding a 1-handle to the page and pre-composing the monodromy with a positive Dehn twist about some closed curve in the page intersecting the co-core of the 1-handle once. If a positive stabilisation of $ (\Sigma, \varphi) $ yields $ (\Sigma', \varphi') $, then the contact manifolds $ (Y, \xi) $ and $ (Y', \xi') $ supported by those respective open books are \emph{contactomorphic}, i.e., there exists a diffeomorphism $ Y \ra Y' $ that induces a map carrying $ \xi $ to $ \xi' $.

	A \emph{Stein surface} is a complex surface $ W $ endowed with a Morse function $ f : W \ra \R $ such that for any non-critical point $ c $ of $ f $, the level set $ f^{-1}(c) $ inherits a contact structure $ \xi_c $, induced by the complex tangencies, that orients $ f^{-1}(c) $ as when $ f^{-1}(c) $ is viewed as the boundary of the complex manifold $ f^{-1}((-\infty,c]) $. We say that a contact manifold $ (Y, \xi) $ is \emph{Stein-fillable} if $ Y $ is orientation-preserving diffeomorphic to such $ f^{-1}(c) $ and $ \xi $ is isotopic to $ \xi_c $. If the 3-manifold is understood, we might simply say that $ \xi $ is Stein-fillable. A necessary condition for $ (Y, \xi) $ to be Stein-fillable is for $ \xi $ to be \emph{tight}, i.e., there being no embedded disc $ D^2 \subset Y $ tangent to $ \xi $ everywhere along $ \pp D^2 $; if there is such a disc, $ \xi $ is \emph{overtwisted}. As noted in the introduction, $ (Y, \xi) $ is Stein-fillable if and only if the monodromy of some open book decomposition of $ Y $ supporting $ \xi $ admits a positive factorisation.


	Having refreshed our memory, we are now ready to tackle the task of constructing Stein-fillable contact manifolds supported by genus one open books whose monodromies do not positively factorise.

	\section{A family of Stein-fillable contact manifolds supported by genus one open books}
	\label{sec:trefoil}

	The goal of this section is to use methods of Conway~\cite{Conway2019} to construct a family of Stein-fillable contact manifolds by surgery techniques, and to determine supporting genus one open book decompositions. Recall that an oriented knot $ K \subset (Y, \xi) $ is \emph{transverse} if its oriented tangent vector is always positively transverse to $ \xi $. By \emph{transverse surgery} on $ K $ we mean an analogue of the usual surgery operation in the contact category, defined by Gay~\cite{Gay-transverse}, in which we first cut out, then re-glue a contact neighbourhood of $ K $ to obtain a new contact manifold; the adjective \emph{`inadmissible'} characterises `adding the twisting' near the knot, while \emph{admissible} transverse surgery `removes the twisting'.

	\subsection{An algorithm for describing open books supporting transverse-surgered manifolds}
	First, we collect necessary ingredients from~\cite{Conway2019} to describe open books supporting the result of inadmissible transverse surgery on a knot that is already a component of the binding of some open book of the original manifold.

	Suppose $ (\Sigma, \varphi) $ is an open book, and $ \Sigma $ has a boundary component $ K $, forming a part of the binding. In the following, by `stabilising $ K $' we mean adding a 1-handle across $ K $ and pre-composing the monodromy with a positive Dehn twist about a curve that is boundary-parallel to one of the two new boundary components. After that, we continue denoting by $ K $ the other boundary component, without a twist about it. Recall that given a rational number $ r < 0 $, we can write it as a negative continued fraction $ [a_1 + 1, a_2, \dots, a_n]^- $, where
	\[
		r = a_1 + 1 - \cfrac{1}{a_2 - \cfrac{1}{\dots - \cfrac{1}{a_n}}}
	\]
	and $ a_i \leqslant -2 $ for all $ i $. The following two propositions give the desired procedure.

	\begin{prop}[{\cite[Proposition 3.9]{Conway2019}}]
	Let $ r \in \Q $ with $ r < 0 $ and $ r = [a_1 + 1, a_2, \dots, a_n]^- $. The open book supporting admissible transverse $ r $-surgery with respect to the page slope on the binding component $ K $ is obtained by, for each $ i = 1, \dots , n $ in order, stabilising $ K $ positively $ | a_i + 2 | $ times and adding a positive Dehn twist about $ K $.
	\end{prop}

	\begin{prop}[{\cite[Proposition 3.12]{Conway2019}}]
	\label{prop:inadmissible}
		Let $ r = p / q \in \Q $ with $ r > 0 $, $ n $ a positive integer such that $ 1 / n < r $, and $ r' = p / (q - np) $. The open book supporting inadmissible transverse $ r $-surgery with respect to the page slope on the binding component $ K $ is obtained by first adding $ n $ negative Dehn twists about $ K $, and then performing admissible transverse $ r' $-surgery on $ K $.
	\end{prop}

	\subsection{Transverse $(+5)$-surgery on a right-handed trefoil}

	Given a simple closed curve $ \sigma $ in a surface $ \Sigma $, denote the positive Dehn twist about $ \sigma $ by $ \tau_\sigma $. Consider a transverse right-handed trefoil knot $ T $ in $ (S^3, \xi_\textrm{std}) $, where $ \xi_\textrm{std} $ is the standard tight contact structure on $ S^3 $. By stabilising the standard open book for $ (S^3, \xi_{\textrm{std}}) $ given by the positive Hopf band, we can take $ T $ to be the binding of an open book $ (\Sigma_{1,1}, \tau_\alpha \tau_\beta) $ with one-holed torus pages supporting $ (S^3, \xi_\textrm{std}) $; this open book is shown on the left of Figure~\ref{fig:monodromy}.

	\begin{center}
		\begin{figure}[h]
			\centering
			\def\svgwidth{0.7\textwidth}
			\makebox[\textwidth][c]{
\begingroup%
  \makeatletter%
  \providecommand\color[2][]{%
    \errmessage{(Inkscape) Color is used for the text in Inkscape, but the package 'color.sty' is not loaded}%
    \renewcommand\color[2][]{}%
  }%
  \providecommand\transparent[1]{%
    \errmessage{(Inkscape) Transparency is used (non-zero) for the text in Inkscape, but the package 'transparent.sty' is not loaded}%
    \renewcommand\transparent[1]{}%
  }%
  \providecommand\rotatebox[2]{#2}%
  \newcommand*\fsize{\dimexpr\f@size pt\relax}%
  \newcommand*\lineheight[1]{\fontsize{\fsize}{#1\fsize}\selectfont}%
  \ifx\svgwidth\undefined%
    \setlength{\unitlength}{228.11375896bp}%
    \ifx\svgscale\undefined%
      \relax%
    \else%
      \setlength{\unitlength}{\unitlength * \real{\svgscale}}%
    \fi%
  \else%
    \setlength{\unitlength}{\svgwidth}%
  \fi%
  \global\let\svgwidth\undefined%
  \global\let\svgscale\undefined%
  \makeatother%
  \begin{picture}(1,0.41030586)%
    \lineheight{1}%
    \setlength\tabcolsep{0pt}%
    \put(0,0){\includegraphics[width=\unitlength,page=1]{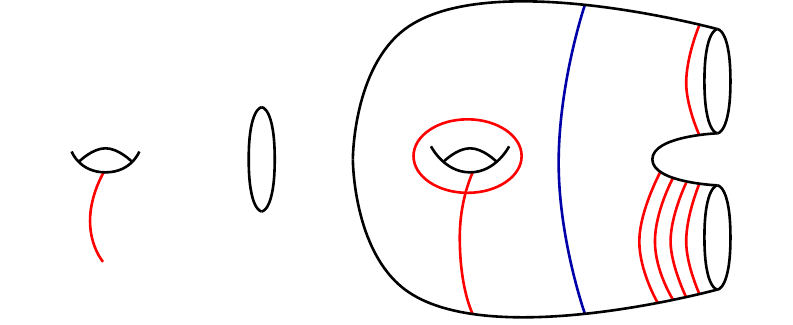}}%
    \put(0.11940806,0.27461248){\color[rgb]{0,0,0}\makebox(0,0)[lt]{\lineheight{1.25}\smash{\begin{tabular}[t]{l}$\alpha$\end{tabular}}}}%
    \put(0.13644514,0.11926244){\color[rgb]{0,0,0}\makebox(0,0)[lt]{\lineheight{1.25}\smash{\begin{tabular}[t]{l}$\beta$\end{tabular}}}}%
    \put(0.57970459,0.27341679){\color[rgb]{0,0,0}\makebox(0,0)[lt]{\lineheight{1.25}\smash{\begin{tabular}[t]{l}$\alpha$\end{tabular}}}}%
    \put(0.59674176,0.09079266){\color[rgb]{0,0,0}\makebox(0,0)[lt]{\lineheight{1.25}\smash{\begin{tabular}[t]{l}$\beta$\end{tabular}}}}%
    \put(0.71577615,0.20571691){\color[rgb]{0,0,0}\makebox(0,0)[lt]{\lineheight{1.25}\smash{\begin{tabular}[t]{l}$\gamma$\end{tabular}}}}%
    \put(0.88607099,0.39327304){\color[rgb]{0,0,0}\makebox(0,0)[lt]{\lineheight{1.25}\smash{\begin{tabular}[t]{l}$\delta_1$\end{tabular}}}}%
    \put(0.88607099,0.00516936){\color[rgb]{0,0,0}\makebox(0,0)[lt]{\lineheight{1.25}\smash{\begin{tabular}[t]{l}$\delta_2$\end{tabular}}}}%
    \put(0,0){\includegraphics[width=\unitlength,page=2]{monodromy.pdf}}%
  \end{picture}%
\endgroup%
}
			\caption{
				On the left: an open book $(\Sigma_{1,1}, \tau_\alpha \tau_\beta)$ supporting $ (S^3, \xi_{\textrm{std}}) $. On the right: an open book $ (\Sigma_{1,2}, \tau_\alpha \tau_\beta \tau_\gamma^{-1} \tau_{\delta_1} \tau_{\delta_2}^4) $ supporting $ (L(5,1), \xi) $, the result of transverse (+5)-surgery on a right-handed trefoil in $ (S^3, \xi_{\textrm{std}}) $.
			} \label{fig:monodromy}
		\end{figure}
	\end{center}

	In the notation of Proposition~\ref{prop:inadmissible}, we have $ r = p/q = 5/1 $, and, choosing $ n = 1 $, we get that $ r' = -5/4 = [-3 + 1, -2, -2 ,-2]^- $. Hence, an open book supporting $ (Y_0, \xi_0) $, the product of inadmissible transverse $(+5)$-surgery on $ T $, is obtained by adding a negative boundary twist $ \tau_\gamma^{-1} $ to the monodromy, stabilising once, adding a twist about $ K $, then adding three more twists about $ K $. Renaming $ K $ to $ \delta_2 $ and the other boundary component to $ \delta_1 $, we conclude that $ (Y_0, \xi_0) $ is supported by the open book $ (\Sigma_{1,2}, \varphi_0) $ shown in Figure~\ref{fig:monodromy} with $ \varphi_0 = \tau_\alpha \tau_\beta \tau_\gamma^{-1} \tau_{\delta_1} \tau_{\delta_2}^4 $. By~\cite[Theorem~1.2]{Hedden-Plam}, $ (Y_0, \xi_0) $ is tight. Figure~\ref{fig:kirby} shows that $ Y_0 $ can be obtained by $ -5 $-surgery on the unknot in $ S^3 $ and thus is diffeomorphic to the lens space $ L(5, 1) $. By work of McDuff~\cite{mcduff90} and Plamenevskaya and Van Horn-Morris~\cite{PlamVHM}, any tight contact structure on $ L(p,1) $ with $ p \neq 4 $ has a unique Stein filling, hence so does $ (Y_0, \xi_0) $.

	\begin{center}
		\begin{figure}[h]
			\centering
			\def\svgwidth{1.2\textwidth}
			\makebox[\textwidth][c]{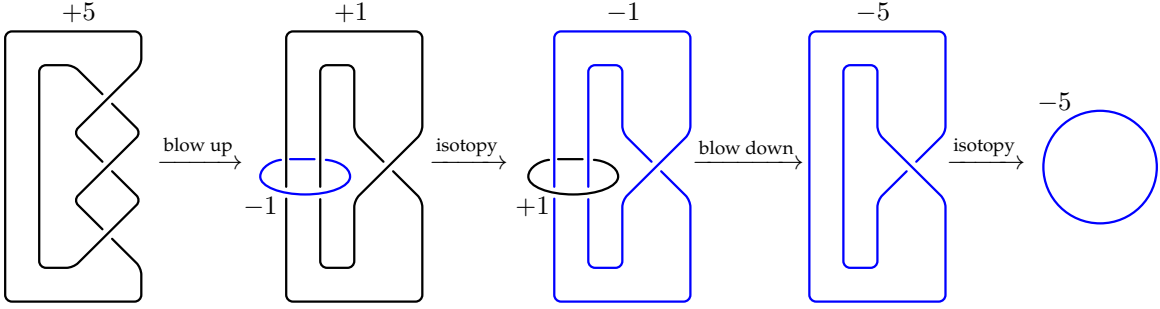}
			\caption{
				A surgery diagram showing that, topologically, the $+5$-surgery on a trefoil and the $-5$-surgery on the unknot give diffeomorphic 3-manifolds.
			} \label{fig:kirby}
		\end{figure}
	\end{center}

	Finally, we observe that $ (Y_n, \xi_n) $, the product of $ n $-fold Legendrian surgery on the $ \delta_2 $ component of $ (\Sigma_{1,2}, \varphi_0) $, is supported by the open book $ (\Sigma_{1,2}, \varphi_n) $, where $ \varphi_n = \tau_\alpha \tau_\beta \tau_\gamma^{-1} \tau_{\delta_1} \tau_{\delta_2}^{4 + n} $. Since Legendrian surgery preserves Stein fillability~\cite{Eliashberg-stein, Weinstein-stein}, this yields an infinite family of Stein-fillable contact manifolds supported by $ (\Sigma_{1,2}, \varphi_n) $ for $ n \geqslant 0 $.

\section{Non-positivity of $ \varphi_n $}
\label{sec:non-positive}

The purpose of this section is to show that the mapping classes $\varphi_n \in \Gamma_{\ourpg}$ do not admit positive factorisations into Dehn twists for all $ n \geqslant 0 $.

Recall that Luo~\cite{Luo}, building on work of Gervais~\cite{Gervais}, showed that the mapping class group of a compact oriented surface admits a presentation in which generators are Dehn twists, and all relations are supported in sub-surfaces homeomorphic to either $\Sigma_{1,1}$ or $\Sigma_{0,4}$. The latter case corresponds to the well-known \emph{lantern relation}, which equates the composition of Dehn twists along curves isotopic to the four boundary components of the sub-surface with a composition of three other twists, illustrated in Figure~\ref{fig:lantern}. Note that if one or more of the boundary curves are homotopically trivial, the relation reduces to the identity. In what follows, given a surface $\Sigma$, we will accordingly refer to any sub-surface homeomorphic to $\Sigma_{0,4}$, none of whose boundary components bound discs in $ \Sigma $, as a \emph{lantern}.

	\begin{center}
		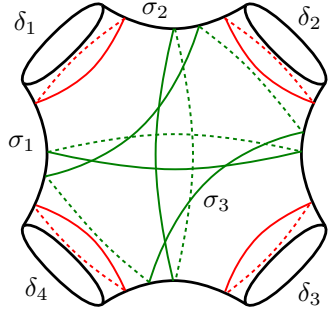
\begin{figure}[h]
			\centering
			\def\svgwidth{0.4\textwidth}
\begingroup%
  \makeatletter%
  \providecommand\color[2][]{%
    \errmessage{(Inkscape) Color is used for the text in Inkscape, but the package 'color.sty' is not loaded}%
    \renewcommand\color[2][]{}%
  }%
  \providecommand\transparent[1]{%
    \errmessage{(Inkscape) Transparency is used (non-zero) for the text in Inkscape, but the package 'transparent.sty' is not loaded}%
    \renewcommand\transparent[1]{}%
  }%
  \providecommand\rotatebox[2]{#2}%
  \newcommand*\fsize{\dimexpr\f@size pt\relax}%
  \newcommand*\lineheight[1]{\fontsize{\fsize}{#1\fsize}\selectfont}%
  \ifx\svgwidth\undefined%
    \setlength{\unitlength}{103.5737706bp}%
    \ifx\svgscale\undefined%
      \relax%
    \else%
      \setlength{\unitlength}{\unitlength * \real{\svgscale}}%
    \fi%
  \else%
    \setlength{\unitlength}{\svgwidth}%
  \fi%
  \global\let\svgwidth\undefined%
  \global\let\svgscale\undefined%
  \makeatother%
  \begin{picture}(1,0.79703759)%
    \lineheight{1}%
    \setlength\tabcolsep{0pt}%
    \put(0,0){\includegraphics[width=\unitlength,page=1]{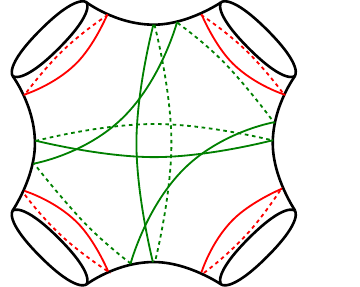}}%
    \put(0.01047526,0.72387753){\color[rgb]{0,0,0}\makebox(0,0)[lt]{\lineheight{1.25}\smash{\begin{tabular}[t]{l}$\delta_1$\end{tabular}}}}%
    \put(0.74907943,0.73835998){\color[rgb]{0,0,0}\makebox(0,0)[lt]{\lineheight{1.25}\smash{\begin{tabular}[t]{l}$\delta_2$\end{tabular}}}}%
    \put(0.74907943,0.01423814){\color[rgb]{0,0,0}\makebox(0,0)[lt]{\lineheight{1.25}\smash{\begin{tabular}[t]{l}$\delta_3$\end{tabular}}}}%
    \put(0.03944015,0.02872038){\color[rgb]{0,0,0}\makebox(0,0)[lt]{\lineheight{1.25}\smash{\begin{tabular}[t]{l}$\delta_4$\end{tabular}}}}%
    \put(-0.00400719,0.41974702){\color[rgb]{0,0,0}\makebox(0,0)[lt]{\lineheight{1.25}\smash{\begin{tabular}[t]{l}$\sigma_1$\end{tabular}}}}%
    \put(0.34357128,0.76732591){\color[rgb]{0,0,0}\makebox(0,0)[lt]{\lineheight{1.25}\smash{\begin{tabular}[t]{l}$\sigma_2$\end{tabular}}}}%
    \put(0.50287797,0.2604395){\color[rgb]{0,0,0}\makebox(0,0)[lt]{\lineheight{1.25}\smash{\begin{tabular}[t]{l}$\sigma_3$\end{tabular}}}}%
  \end{picture}%
\endgroup%

			\caption{
				The lantern relation on $ \Sigma_{0,4} $ is $ \tau_{\delta_1} \tau_{\delta_2} \tau_{\delta_3} \tau_{\delta_4} = \tau_{\sigma_1} \tau_{\sigma_2} \tau_{\sigma_3} $, up to cyclic permutation of $ \tau_{\sigma_i} $ and reordering of $ \tau_{\delta_i} $.
			} \label{fig:lantern}
		\end{figure}
	\end{center}

\vspace{-1.4\baselineskip} We begin with a simple observation. Letting $|\epsilon|_w$ denote the total exponent of $\tau_\epsilon$ in a word $w$ of Dehn twists, we have:

\begin{lem}
\label{lem:boundary_twists}
Let $\delta_1$ and $\delta_2$ denote curves isotopic to the boundary components of $\ourpg$ and let $w$ be a word in Dehn twists about curves on $\ourpg$. Then the number $|\delta_2|_w - |\delta_1|_w$ depends only on the mapping class of $w$.

\end{lem}

\begin{proof}
Using the presentation of Luo in~\cite{Luo}, we see that any non-trivial relation which contains either ${\tau_{\delta_i}}$ must be a lantern relation; the claim follows immediately by showing that every lantern in $\ourpg$ has boundary components isotopic to each $\delta_i$. To see this, let $\Lambda \subset \ourpg$ be any lantern, and $\epsilon$ a curve isotopic to a boundary component of $\Lambda$ but not isotopic to either $\delta_i$. Now, if  $\epsilon$ is non-separating in $\ourpg$, then $\overline{\ourpg\setminus\epsilon}$ is a lantern, so $\Lambda$ is as claimed. On the other hand, if $\epsilon$ is separating, then as it is not boundary-parallel in $\ourpg$ it must cut the surface into $\Sigma_{0,3} \sqcup \Sigma_{1,1}$, neither of which contains a lantern, giving a contradiction.
\end{proof}

We are now ready to prove that $ \varphi_n $ cannot be written as a product of positive twists for any $ n \geqslant 0 $.

\begin{thm}
	The monodromy $ \varphi_n \in \Gamma_{\Sigma_{1,2}} $ represented by the word $ \tau_\alpha \tau_\beta \tau_\gamma^{-1} \tau_{\delta_1} \tau_{\delta_2}^{4 + n} $ does not admit a positive factorisation for any $ n \geqslant 0 $.
\end{thm}
\begin{proof}
	Suppose otherwise, and let $ w $ be be a positive factorisation of $ \varphi_n $. Then $ |\delta_1|_w \geqslant 0 $ and, by Lemma~\ref{lem:boundary_twists}, we have $ |\delta_2|_w \geqslant 3 + n $. Since boundary-parallel Dehn twists commute with all other twists, we can write $ w = w' \tau_{\delta_2}^{3 + n} $ for $ w' $ a positive factorisation of $ \varphi' = \tau_\alpha \tau_\beta \tau_\gamma^{-1} \tau_{\delta_1} \tau_{\delta_2} $.

	Now, following the procedure used in Section~\ref{sec:trefoil}, we recover that $ (\Sigma_{1,2}, \varphi') $, shown in Figure~\ref{fig:phiprime}, supports $ (Y', \xi') $, the result of inadmissible transverse $(+2)$-surgery on a right-handed trefoil.


	\begin{center}
		\begin{figure}[h]
			\centering
			\def\svgwidth{0.4\textwidth}
			\makebox[\textwidth][c]{
\begingroup%
  \makeatletter%
  \providecommand\color[2][]{%
    \errmessage{(Inkscape) Color is used for the text in Inkscape, but the package 'color.sty' is not loaded}%
    \renewcommand\color[2][]{}%
  }%
  \providecommand\transparent[1]{%
    \errmessage{(Inkscape) Transparency is used (non-zero) for the text in Inkscape, but the package 'transparent.sty' is not loaded}%
    \renewcommand\transparent[1]{}%
  }%
  \providecommand\rotatebox[2]{#2}%
  \newcommand*\fsize{\dimexpr\f@size pt\relax}%
  \newcommand*\lineheight[1]{\fontsize{\fsize}{#1\fsize}\selectfont}%
  \ifx\svgwidth\undefined%
    \setlength{\unitlength}{126.86378058bp}%
    \ifx\svgscale\undefined%
      \relax%
    \else%
      \setlength{\unitlength}{\unitlength * \real{\svgscale}}%
    \fi%
  \else%
    \setlength{\unitlength}{\svgwidth}%
  \fi%
  \global\let\svgwidth\undefined%
  \global\let\svgscale\undefined%
  \makeatother%
  \begin{picture}(1,0.72569646)%
    \lineheight{1}%
    \setlength\tabcolsep{0pt}%
    \put(0,0){\includegraphics[width=\unitlength,page=1]{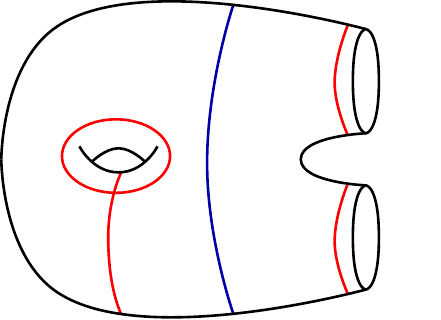}}%
    \put(0.24426689,0.47955621){\color[rgb]{0,0,0}\makebox(0,0)[lt]{\lineheight{1.25}\smash{\begin{tabular}[t]{l}$\alpha$\end{tabular}}}}%
    \put(0.27490143,0.15117995){\color[rgb]{0,0,0}\makebox(0,0)[lt]{\lineheight{1.25}\smash{\begin{tabular}[t]{l}$\beta$\end{tabular}}}}%
    \put(0.48893714,0.35782506){\color[rgb]{0,0,0}\makebox(0,0)[lt]{\lineheight{1.25}\smash{\begin{tabular}[t]{l}$\gamma$\end{tabular}}}}%
    \put(0.79514412,0.69506974){\color[rgb]{0,0,0}\makebox(0,0)[lt]{\lineheight{1.25}\smash{\begin{tabular}[t]{l}$\delta_1$\end{tabular}}}}%
    \put(0.79514412,-0.00028923){\color[rgb]{0,0,0}\makebox(0,0)[lt]{\lineheight{1.25}\smash{\begin{tabular}[t]{l}$\delta_2$\end{tabular}}}}%
  \end{picture}%
\endgroup%
}
			\caption{
				An open book decomposition $ (\Sigma_{1,2}, \tau_\alpha \tau_\beta \tau_\gamma^{-1} \tau_{\delta_1} \tau_{\delta_2}) $ supporting the result of inadmissible transverse $(+2)$-surgery on a right-handed trefoil in $ (S^3, \xi_{\textrm{std}}) $.
			} \label{fig:phiprime}
		\end{figure}
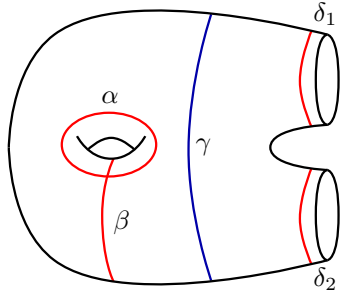
	\end{center}

Denote by $ M(e_0; r_1, r_2, r_3) $ the Seifert fibred space given by the surgery description in Figure~\ref{fig:seifert}. Consider $ M(-1; \frac{1}{2}, \frac{1}{3}, \frac{1}{4}) $. One can verify by sequentially blowing down $ -1 $-framed components that it is orientation-preserving diffeomorphic to $ Y' $, the $ (+2) $-surgery on a right-handed trefoil in $ S^3 $. However, $ (M(-1; \frac{1}{2}, \frac{1}{3}, \frac{1}{4}), \xi) $ is not Stein-fillable for any contact structure $ \xi $ by~\cite[Theorem~1.4]{LiscaLecuona}. By Giroux~\cite{giroux}, it follows that no monodromy of an open book decomposition supporting $ (Y', \xi') $ admits a positive factorisation. Hence no positive factorisation of $ \varphi' $ exists, supplying a contradiction. \qedhere

	\begin{center}
		\begin{figure}[h]
			\centering
			\def\svgwidth{0.4\textwidth}
			\makebox[\textwidth][c]{
\begingroup%
  \makeatletter%
  \providecommand\color[2][]{%
    \errmessage{(Inkscape) Color is used for the text in Inkscape, but the package 'color.sty' is not loaded}%
    \renewcommand\color[2][]{}%
  }%
  \providecommand\transparent[1]{%
    \errmessage{(Inkscape) Transparency is used (non-zero) for the text in Inkscape, but the package 'transparent.sty' is not loaded}%
    \renewcommand\transparent[1]{}%
  }%
  \providecommand\rotatebox[2]{#2}%
  \newcommand*\fsize{\dimexpr\f@size pt\relax}%
  \newcommand*\lineheight[1]{\fontsize{\fsize}{#1\fsize}\selectfont}%
  \ifx\svgwidth\undefined%
    \setlength{\unitlength}{98.96589264bp}%
    \ifx\svgscale\undefined%
      \relax%
    \else%
      \setlength{\unitlength}{\unitlength * \real{\svgscale}}%
    \fi%
  \else%
    \setlength{\unitlength}{\svgwidth}%
  \fi%
  \global\let\svgwidth\undefined%
  \global\let\svgscale\undefined%
  \makeatother%
  \begin{picture}(1,0.62625547)%
    \lineheight{1}%
    \setlength\tabcolsep{0pt}%
    \put(0,0){\includegraphics[width=\unitlength,page=1]{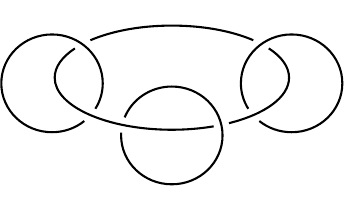}}%
    \put(0.48317835,0.58787022){\color[rgb]{0,0,0}\makebox(0,0)[lt]{\lineheight{1.25}\smash{\begin{tabular}[t]{l}$e_0$\end{tabular}}}}%
    \put(0.02847614,0.1634814){\color[rgb]{0,0,0}\makebox(0,0)[lt]{\lineheight{1.25}\smash{\begin{tabular}[t]{l}$-1/r_1$\end{tabular}}}}%
    \put(0.4073946,0.0119152){\color[rgb]{0,0,0}\makebox(0,0)[lt]{\lineheight{1.25}\smash{\begin{tabular}[t]{l}$-1/r_2$\end{tabular}}}}%
    \put(0.75599944,0.1634814){\color[rgb]{0,0,0}\makebox(0,0)[lt]{\lineheight{1.25}\smash{\begin{tabular}[t]{l}$-1/r_3$\end{tabular}}}}%
  \end{picture}%
\endgroup%
}
			\caption{
				The Seifert fibred 3-manifold $ M(e_0; r_1, r_2, r_3) $.
			} \label{fig:seifert}
		\end{figure}
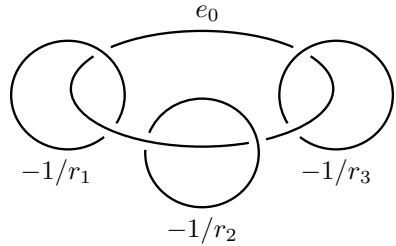
	\end{center}

\end{proof}




\printbibliography
\end{document}